\documentclass[12pt]{article}                     % onecolumn (standard format)
\usepackage{graphicx,epsfig}
\usepackage{amsfonts,color,morefloats,pslatex}
\usepackage{amssymb, amsmath,latexsym,amsthm,a4}

\newtheorem{theorem}{Theorem}
\newtheorem{lemma}[theorem]{Lemma}

\newtheorem{rem}[theorem]{Remark}

\newcommand{\ord}{{\mathrm{ord}}}

\newcommand{\tr}{{\mathrm{Tr}}}

\newcommand{\gf}{{\mathrm{GF}}}
\newcommand{\PG}{{\mathrm{PG}}}

\newcommand{\wt}{{\mathtt{wt}}}

\newcommand{\C}{{\mathcal{C}}}

\newcommand{\E}{{\mathcal{E}}}

\newcommand{\bc}{{\mathbf{c}}}
\newcommand{\be}{{\mathbf{e}}}

\newcommand{\bzero}{{\mathbf{0}}}
\newcommand{\bone}{{\mathbf{1}}}

\begin{document}

\date{}

\title{Maximal arcs and extended cyclic codes}

\author{Stefaan De Winter\thanks{Michigan Technological University,
Houghton, MI 49921, USA},\,
Cunsheng Ding\thanks{The Hong Kong University of Science and Technology, Hong Kong},
and Vladimir D. Tonchev\thanks{Michigan Technological University Houghton, MI 49931, USA}
}

\maketitle

\begin{abstract}
It is proved that for every $d\ge 2$ such that
$d-1$  divides $q-1$, where $q$ is a power of 2,
there exists a Denniston maximal arc $A$ of degree $d$ in $\PG(2,q)$, 
being invariant under a cyclic linear group that fixes one point of $A$
and acts regularly on the set of the remaining points of ${A}$.
Two alternative proofs are given, one geometric proof based on
Abatangelo-Larato's characterization of Denniston arcs, and a second 
coding-theoretical proof based on cyclotomy and the link between maximal 
arcs and two-weight codes.

\end{abstract}

{\bf MSC 2010 codes:} 05B05, 05B25, 51E15, 94B15

\vspace{2mm}
{\bf Keywords:} Maximal arc, 2-design, two-weight code, cyclic code.

\section{Introduction}

Suppose that $P$ is a projective plane of order $q=ds$. A {\it maximal}
$((sd-s+1)d, d)$-{\it arc} (or a maximal arc of degree $d$), 
is a set $A$ of $(sd-s+1)d$ points of $P$ such that every line
of $P$ is ether disjoint from $A$ or meets $A$ in exactly $d$ points
\cite{B}, \cite{Hir}.
The collection of lines of $P$ which have no points in common with $A$
determines a maximal  $((sd-d+1)s, s)$-arc ${A}^{\perp}$
(called a {\it dual arc}) in the dual plane ${P}^{\perp}$.
A {\it hyperoval} is a maximal arc of degree 2.

Maximal arcs of degree $d$ with $1< d < q$ do not exist 
in any Desarguesian plane of odd order
$q$ \cite{BBM},
and are known to exist in every Desarguesian plane of
even order (Denniston \cite{Den}, Thas \cite{Thas74}, \cite{Thas80};
see also  \cite{DWM}, \cite{HM1}, \cite{HM2}, \cite{Math}),
 as well as in some non-Desarguesian planes of even order
\cite{H1}, \cite{H2}, \cite{H3}, \cite{H4},
\cite{HST}, \cite{PRS}, \cite{Thas74}, \cite{Thas80}.

In \cite{AL} Abatangelo and Larato proved that a maximal arc
$A$ in $\PG(2,q)$, $q$ even, is a Denniston arc (that is, $A$ 
can be obtained via Denniston's construction \cite{Den}), if and only if $A$ is 
invariant under a linear collineation of $\PG(2,q)$, being a cyclic group of order $q+1$.
Collineation groups of maximal arcs in $\PG(2,2^t)$ are
further studied in \cite{HP}.

Abatangelo-Larato's characterization of Denniston's
arcs implies, in particular, that a regular 
hyperoval $\cal{H}$ in
$\PG(2,2^t)$ is characterized by the property that $\cal{H}$ is stabilized by
a cyclic collineation group of order $q+1$ that fixes one point of $\cal{H}$
and acts regularly on the remaining $q+1$ points of $\cal{H}$.
Consequently, the two-weight $q$-ary code associated with  $\cal{H}$ (cf.
\cite{CK}), is an extended cyclic code.

The subject of this paper is a class of maximal arcs that generalize 
this property of regular hyperovals.
It is proved that for every $d\ge 2$ such that
$d-1$  divides $q-1$, where $q$ is a power of 2, 
there exists a maximal arc $A$ of degree $d$ in $\PG(2,q)$ that 
is invariant under a cyclic linear group that fixes one point of $A$ 
and acts regularly on the set of the remaining points of $A$,
hence, the two-weight code $C$ associated with $A$ is an extended cyclic code.
Two alternative proofs are given, one geometric proof based on
Abatangelo-Larato's characterization of Denniston arcs, and a coding-theoretic proof
based on cyclotomy.

\section{Maximal arcs with a cyclic automorphism group}

\begin{theorem}
\label{t1}
  Let $q=2^{km}$ and  $d=2^m$, ($m, k \ge 1$).
There exists a partition of AG$(2,q)$
into $\frac{q-1}{d-1}$ maximal Denniston arcs of degree $d$ sharing a unique point,
 and such that there is a cyclic group $G$ acting sharply transitively on the points
 of each of the arcs distinct from the nucleus.
\end{theorem}

\begin{proof}
 Assume $x^2+bx+1$ is an irreducible quadratic form over $\mathbb{F}_q$, 
and let $F_l$, $l\in \mathbb{F}_q\cup \{\infty\}$, be the conic in PG$(2,q)$ with equation
 $x^2+bxy+y^2+lz^2=0$. It is clear that $F_0$, the point $(0,0,1)$ is the nucleus of each
 of the $q-1$ nondegenerate conics $F_l$, $l\in \mathbb{F}_q^*$, and let $F_\infty$ be
 the line $z=0$. We will partition the affine plane AG$(2,q)=\PG(2,q)\setminus (z=0)$.

Let $\mathbb{F}_d$ be the unique subfield of order $d$ of $\mathbb{F}_q$. 
Let $H$ be the additive group of $\mathbb{F}_d$. By Denniston's construction of maximal arcs
\cite{Den},  it follows that $A=\cup F_l$, $l\in H$, is a maximal arc of degree $d$.

We will show that $A$ admits a cyclic group of automorphisms acting sharply transitively
 on the points of the arc distinct from the nucleus. Consider the following group:

$$ G=\left\{ \left( \begin{array}{ccc} \alpha+a\beta & \beta & 0\\ \beta & \alpha & 0\\ 0 & 0 & \gamma    \end{array}  \right): \alpha,\beta\in \mathbb{F}_q, \alpha^2+a\alpha\beta+\beta^2=1, \gamma\in \mathbb{F}_d^* \right\}.$$

This group is the direct product of
 $$G_1= \left\{ \left( \begin{array}{ccc} \alpha+a\beta & \beta & 0\\ \beta & \alpha & 0\\ 0 & 0 & 1    \end{array}  \right): \alpha,\beta\in \mathbb{F}_q, \alpha^2+a\alpha\beta+\beta^2=1 \right\},$$
and
 $$G_2=\left\{ \left( \begin{array}{ccc} 1 & 0 & 0\\ 0 & 1 & 0\\ 0 & 0 & \gamma    \end{array}  \right):  \gamma\in \mathbb{F}_d^* \right\}.$$

By a result of Abatangelo and Larato \cite{AL} $G_1$ is a cyclic group of order $q+1$
 acting sharply transitively on the points of each of the conics $F_l$, $l\in \mathbb{F}_q^*$.
 On the other hand it is clear that $G_2$ is a cyclic group of order $d-1$ that acts
 transitively on the set of conics $F_l$, $l\in H\setminus\{0\}$. It follows that $G$ 
is a cyclic group of automorphisms acting sharply transitively on the points of $A$
 distinct from the nucleus.

Next, let $H_1^*=H\setminus\{0\}, H_2^*,\hdots, H_{\frac{q-1}{d-1}}^*$ be the
 (multiplicative) cosets of  $H\setminus\{0\}$ in the multiplicative group
 of $\mathbb{F}_q$. Set $H_i=H_i^*\cup\{0\}$ for all $i$.
We now make the following two observations:

\begin{itemize}
\item $H_i$ is an additive subgroup of order $d$ of the additive group of $\mathbb{F}_q$, 
for all $i\in\{1,\hdots, \frac{q-1}{d-1} \}$;
\item $H_i\cap H_j=\{0\}$ for all $i\neq j$.
\end{itemize}

The first observation follows immediately from the fact that $H$ is an additive
subgroup of $\mathbb{F}_q$, whereas the second observation follows directly from
the fact that $H\setminus\{0\}$ is a subgroup of the multiplicative subgroup of $\mathbb{F}_q$.

For $i\in\{1,\hdots, \frac{q-1}{d-1} \}$ define $A_i$ to be the Denniston maximal
 arc $\cup F_l$, $l\in H_i$.  One easily concludes that the $\frac{q-1}{d-1}$ maximal
 Denniston arcs $A_i$ partition the plane in the desired way.
\end{proof}

\begin{theorem}
\label{t2}
Let $A_i$, $i=1,\hdots, \frac{q-1}{d-1}$ be a set of maximal arcs of degree $d$ sharing a unique
point $P$ and partitioning the point set of AG$(2,q)$. Furthermore assume that there is a 
linear cyclic group $L$ (of order $(d-1)(q+1)$) acting sharply transitively on the points
 of $A_i$, $i=1,\hdots, \frac{q-1}{d-1}$, distinct from $P$. Then the set of maximal arcs
 $A_i$ arises as in Theorem \ref{t1}.
\end{theorem}

\begin{proof}
We assume that AG$(2,q)$ is the affine plane obtained by deleting the line $z=0$ 
 from PG$(2,q)$.  Clearly $A_1$ is invariant under a linear group $C\leq L$ of collineations 
 of PG$(2,q)$ which is cyclic of order $q+1$. It follows from \cite{AL}  
 that $A_1$ (and hence each of the $A_i$) is of Denniston type. Note that this group $C$ 
 of order $q+1$ stabilizes each of the conics in the maximal arc $A_1$. Hence we can assume 
 that the plane is coordinatized in such a way that $A_1$ is contained in the standard pencil 
 with $P=(0,0,1)$. It follows that the group $C$ is the unique cyclic linear group of order 
 $q+1$ stabilizing all conics in the standard pencil, and hence is actually the group $G_1$ 
 from the previous theorem.  Let $H$ be the additive group associated with $A_1$. Without 
 loss of generality we may assume that $1\in H$.  The stabilizer $S$ in $L$ of the line $x=0$  
  clearly has order $d-1$, is cyclic, and fixes the points $P=(0,0,1)$ and $(0,1,1)$. 
 As the orbit of $(0,1,1)$ under $S$ consists of the points $(0,h,1)$, $h\in H\setminus\{0\}$, 
 it follows that $H$ is actually that additive group of the subfield 
 $\mathbb{F}_d\subset\mathbb{F}_q$. Note that this implies that the action of $S$ on all points 
 of the line $x=0$ is known (the action of $S$ on this line corresponds to multiplying the 
 second coordinate of $(0,y,1)$ by a non-zero element of $\mathbb{F}_d$. Also, clearly all 
 $A_i$ are isomorphic.

We next show that all $A_i$, $i>1$, are contained in the standard pencil. Clearly $L$ contains 
 a unique cyclic subgroup $C$ of order $q+1$.  Assume that $A_i$  contains the points $(0,h_i,1)$, 
 $h_i\in H_i$ for some subset $H_i\subset \mathbb{F}_q$ on the line $x=0$. Then, 
whenever $h_i\neq0$, clearly the orbit of $(0,h_i,1)$ under $C$ is a conic in the standard pencil, 
and belongs to $A_i$. It follows that $A_i$ consists of conics contained in the standard pencil.

Now let $H_i$ be the additive subgroup associated with the maximal arc $A_i$, $i>1$. Clearly the 
 set $\{(0,h_i,1):h_i\in H_i\}$ is stabilized by the subgroup $S$ of $L$. It follows that $H_i$ 
 is a multiplicative coset of the additive subgroup $H$. It now easily follows that the set of 
 maximal arcs $A_i$ arises as in the previous theorem, and the group $L$ is actually the
 group $G$ from Theorem \ref{t1}.
\end{proof}

\section{A family of extended cyclic two-weight codes}

It is known that the existence of a maximal $((sd-s+1)d,d)$-arc
in $PG(2,q)$ is equivalent to the existence of a linear
projective two-weight code $L$ over $GF(q)$ of length $(sd-s+1)d$
and dimension 3, having nonzero weights $w_1 =(sd-s)d$
and $w_2 =(sd-s+1)d$  \cite{CK}, \cite{Del}.
If $A$ is a maximal arc of degree $d=2^m$ in $PG(2,2^{km})$ satisfying the 
conditions of Theorem \ref{t1}, the code $L$ is an extended cyclic code.
We will give a coding-theoretical description of this code based on cyclotomy.

Let $m$ and $k$ be positive integers. Define 
\begin{equation}
\label{par}
q=2^{km}, \ d=2^m,  \ n=(q+1)(d-1), \ N=(q-1)/(d-1), \ r=q^2.
\end{equation}
By definition, 
$$ 
N=\frac{r-1}{n}=\frac{q-1}{d-1}=(2^m)^{k-1}+(2^m)^{k-2}+ \cdots + 2^m +1. 
$$
It is straightforward 
to see that $\ord_n(q)=2$. 
Let $\alpha$ be a generator of $\gf(r)^*$. 
Put $\beta=\alpha^N$. Then the order of $\beta$ is $n$. Let $\tr(\cdot)$ 
denote the trace function from $\gf(r)$ to $\gf(q)$. 

The irreducible cyclic code of length $n$ over $\gf(q)$ is defined by 
\begin{eqnarray}\label{eqn-code1}
\C_{(q,2,n)}=\{\bc_a: a \in \gf(r)\}, 
\end{eqnarray} 
where 
\begin{eqnarray}
\bc_a=(\tr(a\beta^0), \tr(a\beta^1), \tr(a\beta^2), \cdots, \tr(a\beta^{n-1}).  
\end{eqnarray}

The complete weight distribution of some irreducible cyclic codes was determined 
in \cite{BM72}. However, the results in \cite{BM72} do not apply to the cyclic
code $\C_{(q,2,n)}$ of \eqref{eqn-code1}, as our $q$ is usually not a prime. 
The weight distribution of $\C_{(q,2,n)}$ is given in the following theorem.

\begin{theorem}\label{thm-oct1}
The code $\C_{(q,2,n)}$ of \eqref{eqn-code1} has parameters $[n,\, 2,\, n-d+1]$ and 
has weight enumerator 
$$ 
1+(q^2-1)z^{(d-1)q}. 
$$
Further, the dual distance of $\C_{(q,2,n)}$ equals $3$ if $m=1$, and $2$ if $m>1$. 
\end{theorem} 

\begin{proof}
Since $q$ is even, $\gcd(q+1, q-1)=1$. It then follows that 
$$ 
\gcd\left(\frac{r-1}{q-1}, \, N \right)= \gcd\left(q+1, \, \frac{q-1}{d-1}\right)=1.  
$$ 
The desired conclusions regarding the dimension and weight enumerator of $\C_{(q,2,n)}$ 
then follow from Theorem 15 in \cite{DY13}.

We now prove the conclusions on the minimum distance of the dual code of $\C_{(q,2,n)}$. 
To this end, we define a linear code of length $q+1$ over $\gf(q)$ by 
\begin{eqnarray}\label{eqn-nov252}
\E_{(q,2,q+1)}=\{\be_a: a \in \gf(r)\}, 
\end{eqnarray} 
where 
\begin{eqnarray}
\be_a=(\tr(a\beta^0), \tr(a\beta^1), \tr(a\beta^2), \cdots, \tr(a\beta^{q})).   
\end{eqnarray} 
Each code $\bc_a$ in $\C_{(q, 2, n)}$ is related to the codeword $\be_a$ in $\E_{(q,2,q+1)}$
as follows: 
\begin{eqnarray}\label{eqn-nov251}
\bc_a=\be_a||\beta^{(q+1)} \be_a|| \beta^{(q+1)2} \be_a|| \cdots || \beta^{(q+1)(d-2)} \be_a, 
\end{eqnarray} 
where $||$ denotes the concatenation of vectors. It is easy to prove 
$$ 
\{\beta^{(q+1)i}: i \in \{0,1, \cdots, d-2\}\}=\gf(d)^* \subseteq \gf(q)^*. 
$$ 
It then follows that $\E_{(q,2,q+1)}$ has the same dimension as $\C_{(q,2,n)}$. 
Consequently, the dimension of $\E_{(q,2,q+1)}$ is $2$, and the dual code 
$\E_{(q,2,q+1)}^\perp$ has dimension $q-1$. It then follows from the Singleton 
bound that the minimum distance $d_E^\perp$ of $\E_{(q,2,q+1)}^\perp$ is at most 
$3$. Obviously, $d_E^\perp \neq 1$. Suppose that $d_E^\perp = 2$. Then there 
are an element $u \in \gf(q)^*$ and two integers $i, j$ with $0 \leq i < j 
\leq q$ such that $\tr(a(\beta^i-u\beta^j))=0$ for all $a \in \gf(r)$. It then 
follows that $\beta^i (1-u \beta^{j-i})=0$. As a result, $\beta^{j-i}=\alpha^{(q-1)(j-i)/(d-1)} =u^{-1} \in \gf(q)^*$, which is impossible, as $0 < j-i \leq q$ and $\gcd(q+1, (q-1)/(d-1))=1$. 
Hence, $d_E^\perp = 3$. Since $\E_{(q,2,q+1)}^\perp$ is a $[q+1, q-1, 3]$ MDS code, 
$\E_{(q,2,q+1)}$ is $[q+1, 2, q]$ MDS code. When $m=1$, we have $d=2$ and hence 
$\C_{(q, 2, n)}=\E_{(q,2,q+1)}$. Consequently, the dual distance of $\C_{(q, 2, n)}$ 
is $3$ when $m=1$. When $m>1$, we have $d-1>1$. In this case, by (\ref{eqn-nov251}) $\C_{(q, 2, n)}^\perp$ has the following codeword 
$$ 
(\beta^{q+1}, \bzero, 1, 0, 0, \cdots,0, 0), 
$$    
which has Hamming weight $2$, where $\bzero$ is the zero vector of length $q$. Hence, 
$\C_{(q, 2, n)}^\perp$ has minimum distance $2$ if $m>1$. This completes the proof. 
\end{proof}

The code $\C_{(q,2,n)}$ is a one-weight code over $\gf(q)$. We need to study the 
augmented code of $\C_{(q,2,n)}$.
% To this end, 
Let $Z(a, b)$ denote the number of solutions 
$x \in \gf(r)$ of the equation 
\begin{eqnarray}
\tr_{r/q}(ax^N)=ax^{N}+a^q x^{Nq}=b, 
\end{eqnarray}
where $a \in \gf(r)$ and $b \in \gf(q)$.

\begin{lemma}\label{lem-mainlemma} 
Let $a \in \gf(r)^*$ and $b \in \gf(q)$. % and let $k \in \{1, 2\}$. 
Then 
\begin{eqnarray*}
Z(a, b)=\left\{ 
\begin{array}{ll}
(d-1)N+1  & \mbox{ if } b = 0, \\
dN \mbox{ or } 0 & \mbox{ if } b \in \gf(q)^*.   
\end{array}
\right.
\end{eqnarray*}
\end{lemma}

\begin{proof} 
Let $\alpha$ be a fixed primitive element of $\gf(q^2)$ as before.
Define $C_{i}^{(N,q^2)}=\alpha^i \langle \alpha^{N} \rangle$ for $i=0,1,...,N-1$, where
$\langle \alpha^{N} \rangle$ denotes the
subgroup of $\gf(q^2)^*$ generated by $\alpha^{N}$. The cosets $C_{i}^{(N,q^2)}$ are
called the cyclotomic classes of order $N$ in $\gf(q^2)$. When $b=0$, it follows from 
Theorem \ref{thm-oct1} that $Z(a, b)=(d-1)N+1$. Below we give a geometric proof of the 
conclusion of the second part. 

We first recall the following natural model for AG$(2,q)$. The points of AG$(2,q)$ are 
the elements GF$(q^2)$, with $0$ naturally corresponding to the point $(0,0)$. 
Let GF$(q)=\{0,\beta_1,\beta_2,\hdots,\beta_{q-1}\}$. The lines of AG$(2,q)$ through
 $(0,0)$ are of the form 
$\{0,\alpha^i\beta_1,\alpha^i\beta_2,\hdots,\alpha^i\beta_{q-1}\}$ 
for $i=0, q-1, 2(q-1),\hdots,q(q-1)$. The rest of the lines of AG$(2,q)$ are translates
of these $q+1$ lines. In this model, multiplication by a non-zero element of GF$(q^2)$
acts as a linear automorphism of AG$(2,q)$ fixing $(0,0)$ and acting fix point free on
 the other points. Hence $C=\{1, \alpha^{q-1}, \alpha^{2(q-1)},\hdots, \alpha^{q(q-1)}\}$ 
is a cyclic group of order $q+1$ acting on $AG(2,q)$. From \cite{AL}, we know that all
 cyclic subgroups of order $q+1$ of PGL$(3,q)$ are conjugate. Hence it follows that the 
orbits of $C$ on AG$(2,q)$ must consist of a unique fixed point (namely $(0,0)$) and $q-1$
 orbits of size $q+1$, each of which is a conic. Now the multiplicative subgroup
 $H=\{\nu_1,\nu_2,\hdots,\nu_{d-1}\}$ of GF$(q^2)$ acts as a group of homologies with
 center $(0,0)$ on AG$(2,q)$. It follows that $C$ acts as the group $G_1$ and $H$ as
 the group $G_2$ from Theorem \ref{t1}. Hence the orbit of the point ``$1$'' under
 the cyclic group $<C,H>$, together with the point ``$0$'', is a maximal arc of degree $d$.
 On the other hand $<C,H>=C_0^{(N,q^2)}$. The desired conclusion then follows. 
\end{proof}

Define 
\begin{eqnarray}\label{eqn-code2}
\widetilde{\C}_{(q,2,n)}=\{\bc_a +b\bone: a \in \gf(r), \, b \in \gf(q)\}, 
\end{eqnarray} 
where $\bone$ denotes the all-$1$ vector in $\gf(q)^n$. By definition, 
$\widetilde{\C}_{(q,2,n)}$ is the augmented code of $\C_{(q,2,n)}$.  

\begin{theorem}\label{thm-oct2} 
The cyclic code $\widetilde{\C}_{(q,2,n)}$ has length $n$, dimension $3$ and
only the following nonzero weights: 
$$ 
n-d,\ n-d+1, \ n. 
$$ 
The dual distance of $\widetilde{\C}_{(q,2,n)}$ is $4$ if $m=1$, 
and $3$ if $m>1$. 
\end{theorem}

\begin{proof}
By definition, every codeword in $\widetilde{\C}_{(q,2,n)}$ is given by $\bc_a+b\bone$, where 
$a \in \gf(r)$ and $b \in \gf(q)$. By Theorem \ref{thm-oct1}, the codeword $\bc_a+b\bone$ is 
the zero codeword if and only if $(a, b)=(0,0)$. Consequently, the dimension
 of $\widetilde{\C}_{(q,2,n)}$ is $3$. 

When $a=0$ and $b \neq 0$, the codeword $\bc_a+b\bone$ has weight $n$. 
When $a \neq 0$ and $b=0$, by Theorem \ref{thm-oct1}, the codeword $\bc_a+b\bone$
 has weight $n-d+1$. 
When $a \neq 0$ and $b \neq 0$, by Lemma \ref{lem-mainlemma}, the weight of 
the codeword $\bc_a+b\bone$ 
is either $n$ or $n-d$, depending on $Z(a, b)=0$ or $Z(a, b)=dN$. 

The proof of the conclusions on the dual distance of $\widetilde{\C}_{(q,2,n)}$  
is left to the reader.
%, as we do not need the conclusions in the sequel.  
\end{proof}

%\subsection{The extended codes of the augmented codes}

Let $\overline{\widetilde{\C}}_{(q,2,n)}$ denote the extended code
 of $\widetilde{\C}_{(q,2,n)}$. The next theorem gives the parameters
 of this extended code. 

\begin{theorem}\label{thm-oct3}
Let $mk \geq 1$, and let $\overline{\widetilde{\C}}_{(q,2,n)}$ 
be a linear code over $\gf(q)$ with parameters 
$[n+1,\, 3,\, n+1-d]$ and nonzero weights $n+1-d$ and $n+1$. 
Then the weight enumerator of $\overline{\widetilde{\C}}_{(q,2,n)}$ is given by 
\begin{eqnarray}\label{eqn-wtdistnov25}
A(z):=1+ \frac{(q^2-1)(n+1)}{d}z^{n+1-d} + \frac{(q^3-1)d-(q^2-1)(n+1)}{d}z^{n+1}. 
\end{eqnarray} 
Furthermore, the dual distance of the code is $3$ when $m>1$ and $4$ when $m=1$. 
\end{theorem}  

\begin{proof}
By definition, every codeword of $\overline{\widetilde{\C}}_{(q,2,n)}$ is given by 
$$ 
(\bc_a + b\bone, \bar{c}), 
$$
where $\bar{c}$ denotes the extended coordinate of the codeword. Note that 
$\sum_{i=0}^{n-1} \beta^i=0$. We have 
$$ 
\bar{c}=nb=b. 
$$
When $a \neq 0$ and $b=0$, by Theorem \ref{thm-oct1},  
$$ 
\wt((\bc_a + b\bone, \bar{c}))=\wt(\bc_a + b\bone)=n+1-d. 
$$ 
When $a \neq 0$ and $b \neq 0$, by the proof of Theorem \ref{thm-oct2}, 
\begin{eqnarray*}
\wt((\bc_a + b\bone, \bar{c}))= 
\left\{ 
\begin{array}{ll}
n-d+1  & \mbox{ if } Z(a,b)=dN, \\
n+1    & \mbox{ if } Z(a,b)=0. 
\end{array}
\right. 
\end{eqnarray*} 
When $a=0$ and $b \neq 0$, it is obvious that $\wt((\bc_a + b\bone, \bar{c}))=n+1$. 
We then deduce that $\overline{\widetilde{\C}}_{(q,2,n)}$ has only nonzero weights 
$n+1-d$ and $n+1$. 
By Theorem \ref{thm-oct2}, the minimum distance of $\overline{\widetilde{\C}}_{(q,2,n)}^\perp$ is either $3$ or $4$. 
The weight enumerator of $\overline{\widetilde{\C}}_{(q,2,n)}$ is obtained by solvingi
 the first 
two Pless power moments (see also \cite{CK}). 

We now prove the conclusions on the dual distance of $\overline{\widetilde{\C}}_{(q,2,n)}$. 
For simplicity, we put 
$$ 
u=\frac{(q^2-1)(n+1)}{d}z^{n+1-d}, \ v=\frac{(q^3-1)d-(q^2-1)(n+1)}{d}.  
$$ 
By \eqref{eqn-wtdistnov25}, the weight enumerator of $\overline{\widetilde{\C}}_{(q,2,n)}$ 
is $A(z)=1+uz^{n+1-d}+vz^{n+1}$. It then follows from the MacWilliam Identity that the 
weight enumerator $A^\perp(z)$ of $\overline{\widetilde{\C}}_{(q,2,n)}^\perp$ is given by 
\begin{eqnarray}\label{eqn-pt250}
q^3A^\perp(z)
&=& (1+(q-1)z)^{n+1} A\left( \frac{1-z}{1+(q-1)z} \right) \nonumber \\
&=& (1+(q-1)z)^{n+1} + u(1-z)^{n+1-d}(1+(q-1)z)^d    +v(1-z)^{n+1}. 
\end{eqnarray} 
We have 
\begin{eqnarray}\label{eqn-pt251}
(1+(q-1)z)^{n+1}=\sum_{i=0}^{n+1} \binom{n+1}{i} (q-1)^iz^i 
\end{eqnarray}
and 
\begin{eqnarray}\label{eqn-pt252}
v(1-z)^{n+1}= \sum_{i=0}^{n+1} \binom{n+1}{i} (-1)^i v z^i.  
\end{eqnarray} 
It is straightforward to prove that 
\begin{eqnarray}\label{eqn-pt253}
u(1-z)^{n+1-d}(1+(q-1)z)^d = \sum_{\ell=0}^{n+1} 
\left( \sum_{i+j=\ell} \binom{n+1-d}{i} \binom{d}{j}(-1)^i(q-1)^j \right) u z^\ell.    
\end{eqnarray} 

Combining (\ref{eqn-pt250}), (\ref{eqn-pt251}), (\ref{eqn-pt252}) and (\ref{eqn-pt253}), 
we obtain that 
\begin{eqnarray*}
q^3A_1^\perp 
&=& \binom{n+1}{1}[(q-1)-v]+ \\
& & \left[ \binom{n+1-d}{0}\binom{d}{1}(-1)^0(q-1)^1 + \binom{n+1-d}{1}\binom{d}{0}(-1)^1(q-1)^0  \right] u   \\ 
&=& (n+1)[(q-1)-v]+ [d(q-1)-(n+1-d)]u \\ 
&=& 0. 
\end{eqnarray*} 

Combining (\ref{eqn-pt250}), (\ref{eqn-pt251}), (\ref{eqn-pt252}) and (\ref{eqn-pt253}) again, 
we get that 
\begin{eqnarray*}
q^3A_2^\perp 
&=& \binom{n+1}{2}[(q-1)^2+v]+  \binom{n+1-d}{0}\binom{d}{2}(-1)^0(q-1)^2  u +  \\  
& & \binom{n+1-d}{1}\binom{d}{1}(-1)^1(q-1)^1  u + \binom{n+1-d}{2}\binom{d}{0}(-1)^2(q-1)^0  u \\ 
&=& \binom{n+1}{2}[(q-1)^2+v] + \\ 
& & \left[\binom{d}{2}(q-1)^2-(n+1-d)d(q-1)+\binom{n+1-d}{2}  \right]u \\ 
&=& 0. 
\end{eqnarray*} 

Combining (\ref{eqn-pt250}), (\ref{eqn-pt251}), (\ref{eqn-pt252}) and (\ref{eqn-pt253}) the third time, 
we arrive at  
\begin{eqnarray*}
q^3A_3^\perp 
&=& \binom{n+1}{3}[(q-1)^3-v]+ \\
& & \left[\binom{n+1-d}{0} \binom{d}{3}(-1)^0 (q-1)^3 + \binom{n+1-d}{1} \binom{d}{2}(-1)^1 (q-1)^2  \right]u + \\ 
& & \left[\binom{n+1-d}{2} \binom{d}{1}(-1)^2 (q-1)^1 + \binom{n+1-d}{3} \binom{d}{0}(-1)^3 (q-1)^0  \right]u  \\ 
&=& \binom{n+1}{3}[(q-1)^3-v]+ \\
& & \left[ \binom{d}{3} (q-1)^3 - \binom{n+1-d}{1} \binom{d}{2} (q-1)^2  \right]u + \\ 
& & \left[\binom{n+1-d}{2} \binom{d}{1} (q-1) - \binom{n+1-d}{3}   \right]u.  
\end{eqnarray*} 
It then follows that 
\begin{eqnarray*}
6q^3A_3^\perp &=& q^6 d^3 - 4q^6d^2 + 5q^6d - 2q^6 + q^5d^3 -
    3q^5d^2 + 2q^5d - \\ 
    & & q^4d^3 + 4q^4 d^2 - 5q^4 d +
    2 q^4 - q^3 d^3 + 3 q^3 d^2 - 2 q^3 d \\
   &=& (d-2)(d-1)q^3(q^2-1)(qd-q+d). 
\end{eqnarray*} 
Thus, 
\begin{eqnarray}\label{eqn-wtdistw3}
A_3^\perp=\frac{(d-2)(d-1)(q^2-1)(qd-q+d)}{6}. 
\end{eqnarray}

When $m>1$, we have $d>3$. In this case, by \eqref{eqn-wtdistw3} we have $A_3^\perp>0$. 
When $m=1$, by \eqref{eqn-wtdistw3} we have $A_3^\perp=0$. As a result, the dual distance 
is at least $4$ when $m=1$. On the other hand, the Singleton bound tells us that the dual
 distance is at most $4$ when $m=1$. Whence, the dual distance must be $4$ when $m=1$.

Thus, in all cases, the extended code $\overline{\widetilde{\C}}_{(q,2,n)}$
is projective, hence is associated with a maximal $(n+1,d)$-arc in $PG(2,q)$.

% This completes the proof.   
\end{proof}

\begin{theorem}
\label{t9}
If $mk>1$, the supports of the codewords with 
weight $n+1-d$ in $\overline{\widetilde{\C}}_{(q,2,n)}$ 
form a $2$-design $D$ with parameters 
$$ 
2-\left(n+1, \ n+1-d, \  \frac{(n+1-d)(n-d)}{d(d-1)}\right). 
$$
\end{theorem} 

\begin{proof} 
The supports of the codewords of weight $n+1-d$ in 
$\overline{\widetilde{\C}}_{(q,2,n)}$ form a $2$-design
by the Assmus-Mattson theorem \cite{AM} 
%(cf., e.g \cite[2.6]{T88}). 
Since $n+1-d$ is the minimum 
distance of the code, the total number of blocks in the design is given by 
$$ 
\frac{(q^2-1)(n+1)}{(q-1)d}=\frac{(q+1)(n+1)}{d}.  
$$ 
As a result, 
$$ 
\lambda=\frac{(n+1-d)(n-d)}{d(d-1)}. 
$$
\end{proof} 

\begin{rem}
{\rm
 We note that if $M$ is a $3 \times (n+1)$ generator matrix of the two-weight code
 $\overline{\widetilde{\C}}_{(q,2,n)}$
from Theorem \ref{t9}, the columns of $M$ label the points of a maximal
$(n+1,d)$-arc $A$ in $\PG(2,q)$, and the complementary design $\bar{D}$
of the 2-design $D$ from Theorem \ref{t9} is a Steiner 2-$(n+1,d,1)$ design
having as blocks the nonempty intersections of $A$ with the lines of  $\PG(2,q)$.
}
\end{rem}

\begin{theorem}
\label{t10}
If $m>1$, the supports of the codewords with 
weight $3$ in $\overline{\widetilde{\C}}_{(q,2,n)}^\perp$ form a $2$-design 
with parameters 
$$ 
2-\left(n+1, \ 3, \  d-2 \right). 
$$
\end{theorem} 

\begin{proof}
Let $m>1$. By Theorem \ref{thm-oct3} the code $\overline{\widetilde{\C}}_{(q,2,n)}^\perp$ 
has minimum distance $3$. It follows from the Assmus-Mattson theorem that the supports of 
the codewords of weight $3$ in $\overline{\widetilde{\C}}_{(q,2,n)}^\perp$ form a $2$-design. We then deduce from \eqref{eqn-wtdistnov25} that the number of blocks in this design 
is 
$$ 
b^\perp=\frac{(d-2)n(n+1)}{6}. 
$$ 
Consequently, $\lambda^\perp=d-2$. 
\end{proof}

\section{Acknowledgments}

This material is based upon work that was done while the first author 
was serving at the National Science Foundation.
Vladimir Tonchev  acknowledges support by NSA Grant H98230-16-1-0011.

\end{document}